\documentclass[reqno,12pt,letterpaper]{amsart}
\usepackage{amsmath,amssymb,amsthm,graphicx,mathrsfs,url}
\usepackage[usenames,dvipsnames]{color}
\usepackage[colorlinks=true,linkcolor=Red,citecolor=Green]{hyperref}

\usepackage{graphicx}
\usepackage[colorlinks=true]{hyperref}
\usepackage{color}
\usepackage{amsmath,amsfonts}
\numberwithin{equation}{section}
\usepackage{amsmath, amsthm}
    \usepackage{dsfont}
    \usepackage{fancybox}
    \usepackage{amssymb}

\newcommand{\N}{\mathbb{N}}

\newcommand{\dd}[2]{\dfrac{\partial #1}{\partial #2}}
\newcommand{\disp}[1]{\displaystyle{#1}}

\newcommand{\epsi}{\varepsilon}
\newcommand{\1}{\mathds{1}}
\newcommand{\al}{\al}

\newcommand{\te}{\theta}
\newcommand{\lr}[1]{\langle #1 \rangle}

\newcommand{\PP}{\mathcal{P}}

\newcommand{\LD}{L^2_\Phi}
\newcommand{\Oe}{{\Omega_\epsi}}
\newcommand{\C}{\mathbb{C}}
\newcommand{\oz}{\bar{z}}
\newcommand{\ow}{\overline{w}}
\newcommand{\az}{\alpha}

\newcommand{\lv}{\left\vert}
\newcommand{\rv}{\right\vert}

\newcommand{\CC}{\mathbb C}

\newcommand{\tra}{\mathrm{tr}}
\newcommand{\Ppm}{\Pi_\Phi^M}

\newcommand{\ze}{\zeta}

\title[A quantitative version of Catlin-D'Angelo-Quillen
theorem]{A quantitative version of  Catlin-D'Angelo-Quillen theorem}

\author{Alexis Drouot}
\email{drouot@clipper.ens.fr}
\author{Maciej Zworski}
\email{zworski@math.berkeley.edu}
\address{Department of Mathematics, Evans Hall, University of California,
Berkeley, CA 94720, USA}

\newtheorem{thm}{Theorem}

\newtheorem{lem}{Lemma}[section]

\newtheorem{theorem}[thm]{Theorem}

\begin{document}

\begin{abstract}
Let $f\left(z,\oz\right)$ be a positive bi-homogeneous hermitian form
on $\C^n$, of degree $m$. 
A theorem proved by Quillen and rediscovered by Catlin and D'Angelo states 
that for $N$ large enough, $\lr{z,\oz}^Nf\left(z,\oz\right)$ can be
written as the sum of squares of homogeneous polynomials of degree
$m+N$. We show this works for $ N \geq C_f (( n + m ) \log n )^3 $
where $ C_f $ has a natural expression in terms of coefficients of $ f $, inversely proportional to the minimum of $f$ on the sphere.  The proof
uses a semiclassical point of view on which $ 1/N $ plays a role of 
the small parameter $ h$.

\end{abstract}

\maketitle


\section{Introduction and main result}
Let    $ f = f ( z , \bar z ) $ be a bi-homogeneous form of degree $m \geq 1$ on $\C^n$:
\begin{equation}\label{eq:deff}
   f \left( z, \bar z \right) := \sum_{ | \alpha | = |\beta | = m }
c_{\alpha \beta } z^\alpha \bar z^\beta , \ \ z \in \CC^n ,  \ \
c_{\alpha \beta } \in \C . 
\end{equation}
Here $n \geq 2$, $\az = (\az_1, ..., \az_n) \in \N^n$,  $|\az|: =
\az_1 + ... + \az_n$,  $z^\az :=
z_1^{\az_1} ... z_n^{\az_n}$. 
The following theorem 
 was proved by Quillen in $1968$ \cite{Quillen}, and rediscovered by Catlin and
 D'Angelo in $1996$ \cite{Cad}:

\begin{theorem} \label{thm:qui}
\label{t:cdq}
Suppose $ f $ is given by \eqref{eq:deff} and that 
\[   f ( z ,\bar z ) > 0 , \ \  z \neq 0 . \]
Then there exists $ N_0 $ such that for $ N > N_0 $ 
\begin{equation}
\label{eq:cdq}  \| z \|^{2N} f ( z, \bar z )  = \sum_{ j=1}^{d_N} | P_j^N ( z )
|^2 , \ \ \ \| z \|^2 := \sum_{ j =1 }^n | z_j |^2 , 
\end{equation}
where $ P_j ^N ( z ) $ are homogeneous polynomials of degree $ m + N
$, and $ d_N = \binom{n+m+N}{N} $ is the dimension of the space
of homogeneous polynomials of degree $ m + N $.
\end{theorem}

This result can be considered as the complex variables analogue of Hilbert's
17th problem:
given a multivariate polynomial that takes only non-negative values
over the reals, can it be represented as a sum of squares of rational
functions? The positive answer to this original question was given by 
Artin in $1926$ \cite{Ar}.  For a survey of recent work on the 
hermitian case see the review paper by D'Angelo \cite{Da}.

In this paper we give the following quantitative version of Theorem \ref{t:cdq}:
\begin{theorem}\label{thm:boundN}
Let $ f $ satisfy the assumptions of Theorem \ref{t:cdq} and 
define
\begin{equation}\label{eq:lambda}
\lambda \left( f \right) := \min_{ \| z \| = 1 } f \left( z , \bar z \right) , \ \  
\Lambda  \left( f \right) := \disp{ \left( \sum_{|\az|=|\beta|=m} \left(
      \dfrac{\az ! \beta ! }{m!^2} \right) |c_{\az \beta}
    |^2 \right)^{1/2}}. \end{equation}
Then there exists a universal constant $ C $ such that \eqref{eq:cdq} holds for 
\begin{equation}\label{eq:boundN}
N \geq C \, \frac{\Lambda(f)}{ \lambda(f)} (m+n)^3\log^3 n .
\end{equation}
\end{theorem}


The proofs of Quillen \cite{Quillen} and Catlin-D'Angelo \cite{Cad} are
based on functional analytic methods related to the study of Toeplitz
operators. The existence of $ N_0 $ such that \eqref{eq:cdq} is satisfied is obtained by a non-constructive Fredholm compactness argument  -- see \cite[Section
10]{Puti} for outlines and comparisons of the two proofs, and also
\cite{Da1} for an elementary introduction to the subject.

Here we take a point of view based on the semiclassical study of 
Toeplitz operators -- see \cite[Chapter 13]{Maciej:semi} and
references given there. Our proof of Theorem \ref{thm:boundN} is a quantitative 
version of the proof of Theorem \ref{t:cdq} given in \cite[Section
13.5.4]{Maciej:semi}:  the compactness argument
is replaced by an asymptotic argument with $ N = 1/h $, where 
$ h $ is the semiclassical parameter. The symbol calculus for
Toeplitz operators allows estimates in terms of $ h $ which 
then translate into a bound on $ N$.\\

Better bounds on $ N $ obtained using purely algebraic methods 
already exist and it is an interesting question if such bounds can 
be obtained using semiclassical methods.

In the diagonal (real) case, $ c_{\alpha \beta } = 0 $ if $ \alpha
\neq \beta$, Theorem \ref{t:cdq} is equivalent to a classical theorem of
P\'olya -- see \cite[Section 10.1]{Puti}. In that case a sharp bound on $ N $ was given by Powers and Resnick \cite{PR}:
\begin{equation}\label{eq:PRN}
N > \frac{ m \left( m -1\right)} 2 \dfrac{ \widetilde{\Lambda }\left( f \right) }{ \lambda
  \left( f \right) } - m , 
\ \ \ 
\disp{\widetilde{ \Lambda} \left(f\right) :=  \max_{|\alpha | = m} \left\{ \dfrac{\alpha!}{ m!} |c_{\alpha  \alpha}| \right\}}.
\end{equation} 
It is remarkable that the bound does not depend on the dimension $n$.
To compare this bound to the bound obtained using semiclassical
methods,  we note that 
in the diagonal case, the spectral radius 
used in Lemma \eqref{lem:reg}
is given by $ \widetilde \Lambda ( f ) $.  Hence an easy modification of that lemma leads to the bound
\begin{equation}
N \gtrsim \disp{ \dfrac{\tilde{\Lambda}(f)}{\lambda\left(f\right)} (n+m)^3 \log^3 n},
\end{equation}
which is weaker than the bound \eqref{eq:PRN} from \cite{PR},  roughly
by  a factor of  $ m ( 1 +  n/ m  )^3 $.

In the complex case, To-Yeung \cite[Theorem 1]{ToYe} gave an algebraic 
proof of a better bound than the one provided by our method
in Theorem \ref{thm:boundN}.  They show that 
\[  N \geq n m ( 2m - 1) \frac{ \Lambda^\sharp ( f ) }{ \log 2 \, \lambda
  ( f ) } - n - m , \ \ \ \Lambda^\sharp ( f ) := \sup_{ |z|=1} | f (
z , \bar z ) | . \]

The common feature of all these bounds is the denominator $ \lambda ( f )
$ and the standard example of $ |z_1|^4 + |z_2|^4 - c |z_1|^2 |z_2|^2
$ , $ 0 < c < 2 $, $ z \in \C^2 $ (see for instance \cite[Section
13.5.4]{Maciej:semi}) shows that the $ 1 / \lambda ( f ) $ behaviour
is optimal.

In Putinar's generalization of P\'olya's theorem \cite{Put1}, a much
larger bound was given by Nie and Schweighofer \cite{NS}:
\begin{equation}\label{eq:NSN}
N > c \exp{\left( m^2 n^m \dfrac{\tilde{\Lambda(f)}}{\lambda(f)} \right)^c},
\end{equation}
for some $ c > 0 $.

\medskip

The paper is organized as follows. In Section \ref{preli} we recall
various basic facts about the Bargmann-Fock space and Toeplitz
quantization. Section \ref{basi} presents the basic inequality 
which leads to a bound on $ N$. Section \ref{estE} provides
quantitative estimates on the localization of homogeneous polynomials
in the Bargmann-Fock space, with a stationary phase argument
given in the appendix. The proof of Theorem 
\ref{thm:boundN} is then given in Section \ref{proof}. 

\medskip

\noindent
{\sc Notation.} 
We denote $\lr{x,y}$ for $x,y \in \C^n$ the euclidean quadratic form on $\C^n$ (not the hermitian scalar product):
$ \lr{x,y} := { \sum_{i=1}^n x_i y_i }. $
For $z = \left(z_1, ..., z_n\right) \in \C^n$ we define $\| z \|$ as the standard hermitian norm:
$\| z \|^2 := { \sum_{i=1}^n z_i \overline{z_i}  = \lr{z, \oz}}$.
The measure $dm(z)$ denotes the $2n$-dimensional Lebesgue measure on $\C^n$.
The space of homogeneous polynomials of degree $M$ is denoted
$\PP_M$. 
Finally, for two quantities $A$, $B$, we write $A \gtrsim B$, if there exist a (large, universal) constant $C$, such that $A \geq C B$. 

\medskip
\noindent
{\sc Acknowledgments.} We would like to thank Mihai Putinar for
encouraging us to pursue this project and John D'Angelo for providing
useful comments and and references, especially reference \cite{ToYe}. 
The partial support by 
the National Science
Foundation for partial support under the grant DMS-1201417 is also
gratefully
acknowledged.

\section{ Preliminaries: Bargmann-Fock space and Toeplitz quantization}
\label{preli}

Quillen's original proof of Theorem \ref{t:cdq} used the Bargmann-Fock
space -- see \cite[Section 10]{Puti},\cite{Quillen} and \cite[Section
13.5.4]{Maciej:semi}.  We modify it slightly by introducing a
semiclassical parameter $ h$ and considering the subpace of homogeneous
polynomials of degree $M $, $ {\mathcal P}_M$.

A Hilbert space {\em Bargmann-Fock} norm on $ {\mathcal P}_M $ is
given by 
\begin{equation*}
\| u \|_{\PP_M}^2 = \disp{ \int_{\C^n} |u(z )|^2 e^{-\|z\|^2/h} dm\left(z\right) }
\end{equation*}
and we can extend this norm to any function $u$ such that 
\begin{equation*}
\disp{ \int_{\C^n} |u\left(z,\oz\right)|^2 e^{-\|z\|^2/h} dm\left(z\right) < \infty}.
\end{equation*}
We denote the resulting space by  $\LD$. 
The closed subspace of holomorphic functions is denoted by $H_\Phi$.
The measure $\exp(-\|z \|^2/h) dm ( z ) $ will sometimes be written as
$dG(z)$.

The Bergman projector $\Pi_\Phi$, is the orthogonal projector $ \LD \to
H_\Phi$ and 
to compute it we identify an orthonormal basis of $H_\Phi$. The
following standard lemma is a rephrasing of \cite[Theorem 13.16]{Maciej:semi}:

\begin{lem}
Let us define
\begin{equation}\label{eq:falpha}   f_\alpha \left( z \right) :=  \frac{ 1 } { \left( \pi h \right)^{n/2} } \left( \frac{ 1 } { h^{|\alpha|}  { \alpha !}
  }\right)^{1/2} z^\alpha.
\end{equation}
Then 
\begin{enumerate}
\item[$\left(i\right)$] The set of $ f_\alpha $'s is an othonormal basis on $H_\Phi$.
\item[$\left(ii\right)$] The Bergman projector $\Pi_\Phi$ can be written
\begin{equation*}
\Pi_\Phi u\left(z\right) = \disp{\int_{\C^n} \Pi\left(z,w\right) u\left(w\right) dm\left(w\right)}
\end{equation*}
where
\begin{equation*}
\Pi\left(z,w\right) := \dfrac{1}{\left(\pi h\right)^n} \exp \left( \dfrac{1}{h}\left(\lr{z,\overline{w}} - |w|^2\right) \right).\\
\end{equation*}
\end{enumerate}
\end{lem}


To connect the study of positive bi-homogeneous forms to Bargmann-Fock
space, we recall
the standard result (see \cite[Lemma 13.17]{Maciej:semi}):

\begin{lem}\label{Sums}
A bi-homogeneous form of degree $m $ can be written 
as a sum of squares of homogeneous polynomials, 
\begin{equation*}
 f \left( z, \bar z \right) = \sum_{ j =1}^k | P_j \left( z \right) |^2 , \ \ 
P_j \left( z\right) = \sum_{ | \alpha | = m } p_\alpha^j z^\alpha ,
\end{equation*}
if and only if the matrix $ \left( c_{\alpha \beta } \right)_{|\az|=|\beta|=m} $ is positive
semidefinite.
\end{lem}

Thus to prove Theorem \ref{t:cdq} we need  to show that the matrix of the hermitian form $\lr{z,\oz}^N f\left(z,\oz\right)$ is positive for $N$ large enough. Let us compute this matrix. Since 
\begin{equation*} 
\frac{ \langle z , \bar z \rangle^N } {N!} = \sum_{| \mu | = N }
\frac{ z^\mu \bar z ^\mu}{ \mu!}, 
\end{equation*}
\begin{align*}    
 \lr{ z , \bar z }^N f \left( z, \bar z \right) & = 
 \sum_{ | \alpha| = |\beta| = m\atop | \mu | = N } \frac{ c_{\alpha \beta} } {
  \mu!}   z^{ \alpha + \mu } \bar z^{ \beta + \mu } 
 = \sum_{ |\gamma| =
  |\rho| = m + N } c_{\gamma  \rho}^N z^{\gamma } \bar z ^{\rho},
\end{align*}
where
\begin{equation}
\label{eq:cabN}
c_{\rho \gamma}^N = { \sum_{\substack{{ \az+\mu = \rho}  \\
      {\beta+\mu = \gamma, |\mu|=N}}}  \frac{ {c_{\az \beta}}} {{\mu!}
  }}  ,  \ \ | \rho | = | \gamma | = N + m .  
\end{equation}
The following essential idea comes from the work of Quillen in
\cite{Quillen}. It relates the positivity of the matrix
\eqref{eq:cabN}  to the positivity of a differential operator. 

 Let $P_f$ be the following differential operator
\begin{equation}\label{eq:Qu}
P_f = \disp{\sum_{ |\alpha| = |\beta | = m } c_{\alpha \beta } z^\alpha \left(
h \partial_z\right)^\beta} : H_\Phi \longrightarrow H_\Phi .
\end{equation}
Since $f$ is real, $\overline{c_{\az\beta}} = c_{\beta \az}$. Thus the
formula \eqref{eq:repr2} shows that $P_f$ is self adjoint. Let us
explain now how the positivity condition and the operator $P_f$ are
related. 

A simple calculation (see \cite[Section 13.5.5]{Maciej:semi}) based on the definition and \eqref{eq:repr1}
shows that for all $u,v \in \PP_{m+N}$, 
\begin{equation*}
\lr{P_f u, v}_{\PP_{m+N}} = \pi^n h^{n+N+2m} \disp{ \sum_{|\gamma|=|\rho|=m+N} \rho! \gamma! c_{\rho \gamma}^N u_\gamma \overline{v}_\rho }
\end{equation*}
where $u_\rho \in \C$, $v_\gamma \in \C$, are given in 
\begin{equation*}
\disp{ u = \sum_{|\rho| = m+N} u_{\rho}z^\rho, \ \ v = \sum_{|\gamma| = m+N} v_{\gamma}z^\gamma }.
\end{equation*}
Thus proving that the matrix  \eqref{eq:cabN} is positive definite is 
equivalent to proving that $P_f$ is a positive operator on
$\PP_{m+N}$.  To make this quantitative we use the 
following lemma which is an application of a more general formula given 
in \cite[Theorem 13.10]{Maciej:semi}:
\begin{lem}
Let $\Pi_\Phi$ be the orthogonal projector from $\LD$ to $H_\Phi$. Then
\begin{equation}\label{eq:repr1}
P_f |_{P_{m+N}} = \disp{\sum_{ |\alpha| = |\beta | = m } c_{\alpha \beta } z^\alpha \Pi_\Phi\left(\oz^\beta \cdot \right), }
\end{equation}
and
\begin{equation}\label{eq:repr2}
P_f |_{P_{m+N}} = \Pi_{\Phi} q\left(z,\oz\right) \Pi_{\Phi}
\end{equation}
where
\begin{equation}\label{eq:defq}
q\left(z,\oz\right) 
= {\sum_{j=0}^m \dfrac{h^j}{j!} \left(-\textstyle{\frac{1}{4}} \Delta \right)^j f\left(z,\oz\right)}.
\end{equation}
\end{lem}

Using \eqref{eq:repr2}, positivity of  $ P_f $  on $ {\mathcal
  P}_{N+m} $ follows from inequality 
\begin{equation*}
\lr{ \Pi_\Phi q \Pi_{\Phi}u, u }_{\PP_{m+N}} \geq c \| u \|^2_{\LD}, \
\ \  u \in {\mathcal P}_{N+m} , 
\end{equation*}
for some constant $ c > 0 $.
But since $\Pi_\Phi u = u$ and $\Pi_\Phi^* = \Pi_\Phi$, it suffices to
prove that for all $u \in \PP_{N+m}$,
with $\LD$-norm equal to $1$, 
\begin{equation}\label{fund}
\lr{ q(z,\oz) u, u }_{\LD} \geq c , 
\end{equation}
and \eqref{fund} is the starting point of our work.

\section{The basic estimate}
\label{basi}


We define the ring
 $\Oe$ as 
\begin{equation*}
\disp{ \Oe := \{ z \in \C^n, 1-\epsi \leq \| z \|^2 \leq 1+\epsi \}. }
\end{equation*} 
For $u \in \PP_{N+m}$ with $\LD$-norm equal to $1$ we have
\begin{align*}
\lr{qu,u}_{L^2_\Phi}& = \disp{ \int_{\C^n} q\left(z,\oz\right) |u\left(z\right)|^2e^{-\|z\|^2/h} dm\left(z\right) }\\
 & = \disp{ \int_{\C^n \setminus \Oe} q\left(z,\oz\right) |u\left(z\right)|^2
  e^{-\|z\|^2/h} dm\left(z\right) }  
+ {\int_{\Oe} q\left(z,\oz\right) |u\left(z\right)|^2e^{-\|z\|^2/h} dm\left(z\right)}\\
 & \geq \disp{\min_{\C^n \setminus \Oe} q \left(\| u \|_{\LD}-\|u\|_{\LD\left(\Oe\right)}^2\right) + \int_{\Oe} q\left(z,\oz\right) |u\left(z\right)|^2e^{-\|z\|^2/h} dm\left(z\right)}\\
 & =  \disp{ \min_{\C^n \setminus \Oe} q \left(1-\|u\|_{\LD\left(\Oe\right)}^2\right) + \lr{qu,u}_{\LD(\Oe)} }.
\end{align*}
Recalling \eqref{eq:defq} we see that
\begin{equation}
\label{eq:0}
\begin{split}
\lr{qu,u}_{\LD(\Oe)} & \disp{ = \sum_{j=0}^m \dfrac{h^j}{j!}
  \int_{\Oe} \dfrac{\left( - {\textstyle{\frac14} \Delta }\right)^j
    f\left(z,\oz\right)}{\| z \|^{2\left(m-j\right)}} |\| z \|^{m-j}
  u\left(z\right)|^2  e^{-\| z \|^2/h}  dm\left(z\right)   }\\
& \disp{ \geq - \sum_{j=0}^m \dfrac{h^j}{j!} \max_{\Oe} \left(
    \dfrac{1}{\| z \|^{2\left(m-j\right)}}\lv \left( 
{{\textstyle{\frac14}} \Delta} \right)^j f\left(z,\oz\right) \rv \right) \|\| z \|^{m-j} u\|_{\LD\left(\Oe\right)}^2 }\\
& \disp{ \geq - \sum_{j=0}^m \dfrac{h^j}{j!}
  E_\epsi\left(h,m+N,m-j\right) \max_{\| z \|=1} \lv\left( {{\textstyle{\frac14}} \Delta}  \right)^j f\left(z,\oz\right)\rv},
\end{split}
\end{equation}
where the quantity $E_\epsi(h,M,k)$ is defined as
\begin{equation}\label{eq:defE}
E_\epsi(h,M,k) := \disp{ \sup_{u \in \PP_M, \| u \|_{\PP_M} = 1} \| \|z\|^k u \|_{\LD(\Oe)} },
\end{equation}
and where we used the  homogeneity of $\Delta^j f$ of degree
$2\left(m-j\right)$. 

Rearranging the terms we obtain
\begin{equation}\label{eq:7}
\begin{split} 
\lr{qu,u}_{L^2_\Phi} \geq & \left(
    \left(1-E_\epsi\left(h,m+N,0\right)\right)\min_{\C^n \setminus \Oe} q
  \right)  
\\
 &  \ \ \ \ \ \ - \sum_{j=0}^m \dfrac{h^j}{j!} E_\epsi\left(h,m+N,m-j\right) \max_{\| z \|=1} \lv\left( {{\textstyle{\frac14}} \Delta} \right)^j f\left(z,\oz\right)\rv .
\end{split} \end{equation}
Moreover,
\begin{align*}
\disp{\min_{1-2\epsi \leq \| z \|^2\leq 1+2\epsi} q\left(z,\oz\right)} & = \disp{ \min_{1-2\epsi \leq \| z \|^2\leq 1+2\epsi} \sum_{j=0}^m \dfrac{h^j}{j!} \left( -{{\textstyle{\frac14}} \Delta} \right)^j f\left(z,\oz\right) }\\
& \geq \disp{ \left(1-2\epsi\right)^m \lambda\left(f\right) - \sum_{j=1}^m \dfrac{h^j}{j!} \left(1+2\epsi\right)^{m-j} \max_{\| z \|=1} \lv\left( {{\textstyle{\frac14}} \Delta} \right)^j f\left(z,\oz\right)\rv }.
\end{align*}
We see that we need an upper bound for $\disp{\max_{\| z \|=1}
  \lv\left( {{\textstyle{\frac14}} \Delta} \right)^j f\left(z,\oz\right)\rv }$
and that is given in the following
\begin{lem}\label{lem:reg} We have the estimate:
\begin{equation}
\label{eq:reg}
\disp{ \max_{\| z \|=1} \lv \left( {{\textstyle{\frac14}} \Delta} \right)^j f\left(z,\oz\right)\rv \leq \left(n m^2\right)^j \Lambda\left(f\right) },
\end{equation}
where $ \Lambda ( f ) $ is defined in \eqref{eq:lambda}.
\end{lem}
To explain the proof we note that since $f$ is a bihomogeneous form of
degree $m$, 
$\Delta^k f$ is a bihomogeneous form of degree $m-k$. 
If we have estimates on $f$, and if we find an explicit relation
between estimates 
on $f$ and $\Delta f$, related to the bound on $\max_{\| z \|=1}
|f\left(z,\oz\right)|$, a recursion procedure will give \eqref{eq:reg}

\begin{proof}
For $ z \in \C^n $ satisfying $ \| z \| = 1 $, put 
$z=\left(r_1 e^{i \te_1},..., r_ne^{i\te_n}\right) $, with $\sum r_i^2 = 1$. Then 
\[ \begin{split}
|f\left(z,\oz\right)| & \leq {\sum_{|\az|=|\beta|=m}  |c_{\az\beta}| r^\az r^\beta }
 \leq  \sum_{|\az|=|\beta|=m} \dfrac{\sqrt{\az ! \beta
       !}}{m!}|c_{\az\beta}|  \sqrt{\dfrac{m !}{ \az !}} r^\az \sqrt{
     \dfrac{m !}{\beta !} } r^\beta  \\
& = \lr{\widetilde{C}R,R } , \ \ \  \ \ 
\widetilde{C} : =\left(\frac{\sqrt{\az ! \beta
      !}}{m!}|c_{\az\beta}|\right)_{|\az|=|\beta|=m}, \ \ R 
:=\left(\sqrt{\frac{m !}{\az!}}r^\az\right)_{|\az|=m} .
\end{split} \]
Since
\begin{equation*}
\| R \|^2 = \disp{\sum_{|\az|=m} \dfrac{m!}{\az !} r^{2\az} = \| r \|^{2m} = 1} 
\end{equation*}
we have
\begin{equation*}
|f\left(z,\oz\right)| \leq  \lr{ \widetilde{C}R,R } \leq \rho(\widetilde{C}) ,
\end{equation*}
where $\rho(\widetilde{C})$ is the spectral radius of $\widetilde{C} $.
The spectral radius can be estimated by $ \Lambda ( f ) $ given in
\eqref{eq:lambda}:
we write $\widetilde{C} = U D U^{-1}$, where $ D $ and  $U$ are
diagonal and orthogonal matrix, respectively. Then
\begin{equation*}
\Lambda\left(f\right)^2 = \tra \left(\widetilde{C}
  \widetilde{C}^*\right) = \tra\left( U D^2 U^{-1} \right) =
\tra\left(
D^2 \right) \geq \rho (\widetilde{C})^2
\end{equation*}
and hence
\begin{equation}\label{eq:max}
\max_{\| z \|=1} |f\left(z, \oz\right)| \leq \Lambda\left(f\right) .
\end{equation}

 We now need to find a relation between $\Lambda\left(f\right)$ and $
 \Lambda\left(\frac{1}{4} \Delta f \right)$. Let $D:=(d_{\gamma
   \rho})$ be the matrix of the bi-homogeneous form $\frac{1}{4}\Delta
 f$, and let us chose $\gamma, \rho$ with
 $|\gamma|=|\rho|=m-1$. Denoting by $ \tilde c_{\alpha \beta} $ the
 entries of $ \widetilde C $ we obtain {
\begin{align}\label{eq:recursion}
\disp{d_{\gamma \rho}} & = {\dfrac{1}{\gamma ! \rho !} \dd{^\gamma}{z^\gamma} \dd{^\rho}{\oz^\rho} \sum_{i=1}^n \dd{}{z^i} \dd{}{\oz^i} f \left(0,0\right)}  =\disp{ \sum_{i=0}^n  \left(\gamma_i +1\right)\left(\rho_i +1\right) c_{\gamma +e_i, \rho +e_i}}\\
   & = \disp{\sum_{i=1}^n \left(\gamma_i +1\right)\left(\rho_i +1\right) \dfrac{m!}{\sqrt{(\gamma+e_i)! (\rho+e_i)!}} \tilde{c}_{\gamma +e_i, \rho +e_i}}\\
   & = \disp{ \dfrac{m!}{\sqrt{\gamma ! \rho !}}\sum_{i=1}^n \sqrt{\left(\gamma_i +1\right)\left(\rho_i +1\right)} \tilde{c}_{\gamma +e_i, \rho +e_i}} \\
   & = \disp{ \dfrac{(m-1)!}{\sqrt{\gamma ! \rho !}}m^2\sum_{i=1}^n  \tilde{c}_{\gamma +e_i, \rho +e_i} }.
\end{align}
If we put
$ \tilde{d}_{\gamma \rho} := {\sqrt{\gamma ! \rho !}} d_{\gamma \rho}/
{(m-1)!} $, and denote the correspoding matrix by $ \widetilde D $, 
then
\begin{equation}
\tilde{d}_{\gamma \rho} \leq m^2\sum_{i=1}^n  \tilde{c}_{\gamma +e_i, \rho +e_i},
\end{equation}}

and
\begin{align}
\Lambda \left( {{\textstyle{\frac14}} \Delta} f \right)^2 & = \disp{ \sum_{|\gamma|=|\rho|=m-1} d_{\gamma \rho}^2 \leq 
 \sum_{|\gamma|=|\rho|=m-1} m^4 \left(\sum_{i=1}^n  \tilde{c}_{\gamma +e_i, \rho +e_i} \right)^2  }\\
\label{eq:5} & \leq \disp{   m^4 n \sum_{|\gamma|=|\rho|=m-1} \sum_{i=1}^n  \tilde{c}_{\gamma +e_i, \rho +e_i}^2  }\\
 & \leq \disp{ n m^4 \cdot n \Lambda\left(f\right)^2 }. 
\end{align}
An easy recursion then gives
\begin{equation*}
\Lambda \left( \left({{\textstyle{\frac14}} \Delta} \right)^j f \right) \leq \left(nm^2\right)^j \Lambda\left(f\right).
\end{equation*}
and inequality \eqref{eq:max} applied to $\left(\frac{1}{4} \Delta \right)^j f$ instead of $f$ proves the lemma.
\end{proof}

The lemma and the lower bound stated after the inequality \eqref{eq:7} imply
\begin{equation}\label{eq:6}
\disp{\min_{1-2\epsi \leq \| z \|^2\leq 1+2\epsi} q\left(z,\oz\right)} \geq \lambda\left(f\right) \disp{ \left(1-2\epsi\right)^m  - \Lambda\left(f\right) \sum_{j=1}^m \dfrac{1}{j!} \left(  nm^2h\right)^j  \left(1+2\epsi\right)^{m-j} }.
\end{equation}
This combined with \eqref{eq:7} leads to the {\em basic  inequality}:
\begin{equation}\label{eq:starting}
\begin{split}
& \lr{qu,u}_{L^2_\Phi} \geq \\
& \ \left(1-E_\epsi\left(h,m+N,0\right)\right) \left( \lambda\left(f\right) \disp{ \left(1-2\epsi\right)^m  - \Lambda\left(f\right) \sum_{j=1}^m \dfrac{1}{j!}  \left( nm^2h \right)^j  \left(1+2\epsi\right)^{m-j} } \right) \\
& \ \ \ \ \ \ \ \ \ - \Lambda\left(f\right) \sum_{j=0}^m \dfrac{1}{j!} \left(
  nm^2h\right)^j E_\epsi\left(h,m+N,m-j\right) .
\end{split}
\end{equation}
All the  work that follows is aimed at finding $h_0$ such that 
for $ h < h_0 $ the right hand side of \eqref{eq:starting} is positive.

\section{Estimates on $E_\epsi$}
\label{estE}

Our goal in this section is to prove that the quantity
$E_\epsi\left(h,M,m\right)$ roughly decreases like
$\exp\left(-M\epsi^2\right) $, under some assumptions relating
$\epsi,h,M,m,n$. It is essentially due to the fact that the
homogeneous polynomials are localised in $\LD$-norm around the sphere
$S^{2n-1} \subset \C^n$, with $1/h \sim M $ -- see \cite[Theorem
13.16, (ii)]{Maciej:semi} for an explanation of this using the
harmonic oscillator. Here we prove
\begin{lem}\label{lem:estE}
Let $\epsi,h,m,n, M>0$ and let us call
\begin{equation}
\sigma := h(M+m+n-1).
\end{equation}
Assume that 
\begin{equation}\label{eq:assepsi}
\dfrac{3}{2} > \sigma >1,  \ \ 1 \geq  \epsi \geq 4(\sigma-1) .
\end{equation}
Then for $ E_\epsilon  $ defined in \eqref{eq:defE} we have
\begin{equation}\label{eq:estEm}
E_\epsi\left(h,M,m\right) \lesssim \disp{ h^m (M+m+n)^{2n+m} \dfrac{1}{\epsi ^2}  \exp\left(-\dfrac{M\epsi^2}{16} \right) }.
\end{equation}
\end{lem}
\begin{proof}
Let $\Pi_\Phi^M$ be the projection from $\LD$ to $\PP_M$.  For $u \in
\PP_M$. 
To estimate the right hand side in \eqref{eq:defE} we note that
\begin{equation*}
\begin{split} 
\| \| z \|^m u \|_{\LD(\Oe)}^2 & = \lr{ u, \Ppm \| z \|^{2m} \1_\Oe
  \Ppm u }_{\LD} \\
& \leq  \| \Ppm \| z \|^{2m} \1_{\Oe} \Ppm \|_{\LD \rightarrow \LD}
\cdot \| u \|^2_{\LD}
\end{split}
\end{equation*} 
 Hence it suffices to 
estimate the norm operator $\| \Ppm \| z \|^{2m} \1_{\Oe} \Ppm \|$, 
and for that we will use the following standard variant of Schur's Lemma:
\begin{lem}\label{Schur}
Let $(X,\mu)$ be a measure space, $K : L^2(X) \rightarrow L^2(X)$ a selfadjoint operator with kernel $k$, that is
\begin{equation*}
Ku(x) = \disp{ \int_X k(x,y) u(y) d\mu(y) }.
\end{equation*}
Assume that there exists an almost everywhere positive function $p$ on $X$ and $\lambda>0$ such that 
\begin{equation}
\label{eq:Schur}
\int_X |k(x,y)| p(y) d\mu(y) \leq \lambda p(x).
\end{equation}
Then $\| K \| \leq \lambda$.
\end{lem}
%

To apply the lemma we first construct the kernel of the projector
$\Ppm = \sum_{|\az|= M} f_\az f_\az^\ast $,
where $f_\az$ was defined in \eqref{eq:falpha}, and $f_\az^*$ is the
linear form $\lr{f_{\az},\cdot}_{\LD}$. 
Writing
\begin{equation*}
\disp{ \Ppm u (z) := \int_{\C^n} \Pi^M(z,w) u(w) e^{ - \| w^2 \| /h}
  dm(w)  }, 
\end{equation*}
we have 
\begin{align*}
\Pi^M(z,w) = & \disp{\sum_{|\az |=M} f_\az(z) \overline{f_\az(w)} }\\
           = & \disp{ \sum_{|\az |=M} \frac{ 1 } { \left( \pi h \right)^n } \left( \dfrac{ 1 } { h^M  \alpha ! }\right) z^\alpha \ow^\az } = \disp{  \frac{ 1 } { \pi^n h^{n+M} } \sum_{|\az |=M} \frac{ 1 } { \alpha !  } z^\alpha \ow^\az } = \disp{ \dfrac{ \lr{z,\ow}^M } { M! \pi^n h^{n+M} } }.
\end{align*}
It follows that the integral kernel of 
 $K = \Ppm \| z \|^{2m} \1_{\Oe} \Ppm$  with respect to the Gaussian
 measure $dG(z):=\exp(-\| z \|^2/h)dm(z)$, $k$,  is given by 
\begin{equation*}
k(z,w) = \disp{\int_{\Oe}  \dfrac{ \lr{z,\overline{\ze}}^M } { M!
    \pi^n h^{n+M} } \dfrac{ \lr{\ze,\ow}^M } { M! \pi^n h^{n+M} } \|
  \ze \|^{2m} dG(\ze) }. 
\end{equation*}
This suggests natural choice of the weight $p = \| z \|^M$  in lemma
\ref{Schur}, 
and  we need to estimate the 
corresponding parameter $\lambda$ in \eqref{eq:Schur}. For that, we need an upper bound on the integral
\begin{equation*}
\disp{ \int_{C^n} |k(z,w)| \| w \|^{M} dG(w) }.
\end{equation*}
An application of the Cauchy-Schwarz inequality inequality gives
\begin{align*}
& \disp{ \int_{C^n} |k(z,w)| \| w \|^{M} dG(w)}  \leq \disp{ \int_{\C^n} \int_{\Oe} \| w \|^M    \dfrac{ \| z \|^M \| \ze \|^M } { M! \pi^n h^{n+M} } \dfrac{ \| \ze \|^M \| w \|^M } { M! \pi^n h^{n+M} } \| \ze \|^{2m} dG(\ze) dG(w) }\\
      & \ \ \ \leq \disp{ \| z \|^M \left( \int_{\C^n} \dfrac{\| w \|^{2M}}{ M! \pi^n h^{n+M} } dG(w) \right) \left( \int_{\Oe}     \dfrac{ \| \ze \|^{2M+2m} } { M! \pi^n h^{n+M} } dG(\ze) \right)}.
\end{align*}
Thus it is sufficient to estimate the following integrals:
\begin{equation}\label{eq:sep12}
I_1 = { \int_{\C^n} \dfrac{\| w \|^{2M}}{ M! \pi^n h^{n+M} } dG(w) } ,
\ \  \ \ 
I_2 = 
{ \int_{\Oe}     \dfrac{ \| \ze \|^{2M+2m} } { M! \pi^n h^{n+M} } dG(\ze) }.
\end{equation}
A polar coordinates change of variables, followed by a substitution $
t  = r^2/ h $,   gives
\begin{equation} \label{eq:21} \begin{split}
I_1 & = \dfrac{|S^{2n-1}|}{M! \pi^n h^{n+M}} \int_0^\infty     r^{2M+2n-1} e^{-r^2/{h}}  dr 
 = \disp{\dfrac{|S^{2n-1}|}{2 M! \pi^n h^{n+M}} h^{n+M} \int_0^\infty     t^{M+n-1} e^{-t} dt}  \\ 
& = \dfrac{(M+n-1)!}{ M! (n-1)!} \leq \binom{M+n}{n} 
\leq (M+n)^n , 
\end{split} 
\end{equation} 
where $ |S^{2n-1}| = {2\pi^n}/{(n-1)!}$ denotes the volume of the $2n-1$ 
dimensional sphere.

Turning to $ I_2 $ in \eqref{eq:sep12} we make two changes of
variables, $z=r\te$, then $r^2 = t$, so that 
\begin{equation}
\label{eq:I2}
\begin{split}
I_2 & \disp{ = |S^{2n-1}| \int_{r^2 \notin [1\pm \epsi]} \dfrac{r^{2M+2m+2n-1}}{M! \pi^n h^{n+M}} \exp\left( -\dfrac{r^2}{h} \right) dr }\\
& \disp{ =\dfrac{|S^{2n-1}|}{2M! \pi^n h^{n+M}}\int_{t \notin [1\pm \epsi]} t^{M+m+n-1} e^{-t/h} dt }.
\end{split}
\end{equation}
The last integral is very close to the integral appearing in the
following lemma which will be proved in the appendix:
\begin{lem}\label{lem:estJ}
Let $\rho>0, \delta < 1$. We define
\begin{equation*}
J\left(\rho,\delta\right) := \disp{ \int_{t \notin [1-\delta, 1+\delta]} t^\rho e^{ -\rho t} dt .}
\end{equation*}
Then
\begin{equation}\label{eq:estJ1}
J\left(\rho,\delta\right) \lesssim \dfrac{1}{\rho \delta^2} \exp \left(-\rho\left(1+\dfrac{\delta^2}{4}\right)\right).
\end{equation}
\end{lem}
To apply this lemma to the last integral in \eqref{eq:I2} 
we make the change of variable $t/h = (M+m+n-1)s$. 
To assure that the  interval of integration does not change much, we
claim that under assumptions of Lemma \ref{lem:estE} we have, 
\begin{equation}\label{eq:assumption}
[1 \pm \epsi /2] \subset \dfrac{1}{h(M+m+n-1)} [1 \pm \epsi] = \dfrac{1}{\sigma} [1 \pm \epsi].
\end{equation}
Indeed, \eqref{eq:assepsi} implies the following inequalities:
\begin{equation}\label{eq:p12}
1- \dfrac{\epsi}{2} \geq \dfrac{1}{\sigma} (1-\epsi), \ \ \ 
1+ \dfrac{\epsi}{2} \leq \dfrac{1}{\sigma} (1+\epsi).
\end{equation}
The first one is straightforward, since it is equivalent to $2\sigma-2
\geq (\sigma-2) \epsi$, and $\sigma-2 <0$. 
The second inequality in \eqref{eq:p12} is equivalent to
$ {(2\sigma-2)}/({2-\sigma}) \leq \epsi $,  so that in view of 
\eqref{eq:assepsi} we need to check that
$ {(2\sigma-2)}/({2-\sigma}) \leq 4(\sigma -1) $
which follows from the assumption 
$\sigma < 3/2$. 

Returning to \eqref{eq:I2} we have 
\begin{equation*}
\disp{\int_{t \notin [1\pm \epsi]  } t^{n+m+M-1} e^{-t/h} dt \leq [h(M+m+n-1)]^{M+m+n}\int_{ s \notin [1 \pm \epsi/2]} (te^{-t})^{M+m+n-1} dt }.
\end{equation*}
Applying Lemma \ref{lem:estJ} gives
\begin{align*}
\disp{\int_{t \notin [1\pm \epsi/2]  } t^{n+m+M-1} e^{-t/h} dt} &
\lesssim \disp{  \dfrac{[h(M+m+n-1)]^{M+m+n}}{(M+m+n-1)\epsi^2} 
e^{-(M+n+m-1) (1+{\epsi^2}/{16} )   }} \\
& \lesssim \disp{ \dfrac{[h(M+m+n)]^{M+m+n}}{\epsi^2} e^{-(M+n+m) (1+{\epsi^2}/{16} ) } }.
\end{align*}
Hence 
\begin{equation}
\label{eq:16}
\begin{split}
I_2 
 & \lesssim \disp{[h(M+m+n)]^{M+m+n} \dfrac{|S^{2n-1}|}{2\pi^n}
   \dfrac{e^{-M-m-n}}{h^M M! }  \dfrac{1}{h^n\epsi^2} 
e^{ -{M\epsi^2}/{16}  } } \\
 & \lesssim \disp{ h^m (M+m+n)^{M+m+n}
  \dfrac{e^{-M-n-m}}{M! (n-1)! }  \dfrac{1}{\epsi^2}
e^{ -{M\epsi^2}/{16}  }  }
\end{split}
\end{equation}
To simplify the upper bound in \eqref{eq:16} 
we first use Stirling's formula to obtain (with a small irrelevant loss since
$ k^k \lesssim k! e^k / \sqrt k $)
\begin{equation*}
\disp{ (M+m+n)^{M+m+n} \lesssim (M+m+n)! e^{M+m+n} } . 
\end{equation*}
Thus the bound in \eqref{eq:16} can be replaced by 
\begin{align*}
I_2 \lesssim \disp{ h^m \dfrac{(M+m+n)!}{M!} \dfrac{1}{\epsi ^2}
  e^{- M \epsilon^2/16} } 
\lesssim \disp{ h^m (M+m+n)^{m+n} \dfrac{1}{\epsi ^2}    e^{- M
    \epsilon^2/16}   }.
\end{align*}
Combining this with the bound \eqref{eq:21},  and applying Lemma
\ref{Schur} gives
\begin{align*}
\| K \| & \lesssim \disp{ h^m (M+n)^n (M+m+n)^{m+n} \dfrac{1}{\epsi ^2}e^{- h^{-1/3}/16}  }\\
  & \lesssim \disp{ h^m (M+n+m)^{2n+m} \dfrac{1}{\epsi ^2}  e^{-      h^{-1/3}/16}  }.
\end{align*}
This completes the proof of Lemma \ref{lem:estE}.
\end{proof}

\section{Proof of Theorem \ref{thm:boundN}}
\label{proof}

We now combine the basic inequality
\eqref{eq:starting} with the estimate on $E_\epsi$ given in Lemma
\ref{lem:estE}. 
We split \eqref{eq:starting} into four terms:
\begin{enumerate}
\item[$\left(i\right)$] ${A_0}=\lambda\left(f\right) \disp{
    \left(1-2\epsi\right)^m}$ which is the leading term; 
\item[$\left(ii\right)$] ${A_1}=\disp{\lambda\left(f\right)
    E_\epsi\left(h,m+N,0\right)  \left(1-2\epsi\right)^m}$ decreases 
exponentially to $0$ as $h\rightarrow 0$;
\item[$\left(iii\right)$] ${A_2} = \disp{\Lambda\left(f\right) \sum_{j=1}^m
    \dfrac{1}{j!}  \left( nm^2h\right)^j
    \left(1+2\epsi\right)^{m-j}}$ will be estimated by noting that it
  is dominated by its first term;
\item[$\left(iv\right)$] ${A_3}=\disp{\Lambda\left(f\right) \sum_{j=0}^m
    \dfrac{1}{j!} \left( nm^2h\right)^j
    E_\epsi\left(h,m+N,m-j\right)}$ will require more care but decreases 
exponentially to $0$ as $h\rightarrow 0$.
\end{enumerate}

We want to optimize the parameters $h, M, \epsi$ as functions of the
order of $ f$,  $m$, and the dimension $ n $. 
We aim to show that ${A_0} \gg {A_1}, {A_2}, {A_3}$, using Lemma
\ref{lem:estE}. For this we need to check that the assumption \eqref{eq:assepsi} is satisfied. 

The basic strategy is outlined as follows
\begin{itemize}
\item \eqref{eq:assepsi} is satisfied if for all $0\leq j \leq m$, $h^{-1} \sim N+2m+n-j$ and $h(N+2m+n-j) \leq 1$. Thus we need $h^{-1} \sim N \gg m,n$.
\item ${A_0} \gtrsim {A_1}$: we want to apply Lemma \ref{lem:estE} and thus we need $\epsi^2/h \geq -n\log (h)$;
\item ${A_0} \gtrsim {A_2}$: for this to hold ${A_2}$ has to be greater than
  the first term of the sum in ${A_2}$, 
$nm^2(1+2\epsi)^{m-1} h$; thus the term 
$(1+2\epsi)^m$ has to remain bounded as $m \rightarrow \infty$: we need $\epsi \lesssim 1/m$.
\item ${A_0} \gtrsim {A_3}$: the term ${A_0}$ has -- at least -- to be greater
  than the first term of the sum in ${A_3}$ ; thus we need to have 
${A_0} \gtrsim E_\epsi(h,N+m,m)$; using Lemma \ref{lem:estE}, this holds
when $\epsi^2/h \geq -(n+m)\log (h)$.
\end{itemize}
We define $ \epsi $ as $\epsi=h^a$, where $ a $ will be determined.
From the considerations above  we need 
\[ h^{2a-1} \gtrsim (n+m) \log \frac 1 h  \ \ \text{and} \ \ h^a \lesssim
1/m. \]
To express this as one condition, we demand $a=1-2a$, that is,
$a=1/3$. This leads to the necessary relations:
\begin{equation}\label{eq:assump}
\epsi = h^{-1/3}/16, \ \ h \lesssim (n+m)^{-3}, \ \ N=h^{-1}.
\end{equation}

\subsection*{Application of the estimates on $E_\epsi$.} To use 
estimates on 
$E_\epsi(h,m+N,m-j)$ for $0 \leq j \leq m$ we need the
assumption \eqref{eq:assepsi}  to hold. That means that
\begin{equation}\label{eq:assump1}
1 <h(N+2m+n-j-1) < \dfrac{3}{2}, \ \ 1 \geq \epsi \geq 4(h(N+2m-j+n)-1).
\end{equation}
Since $ N = 1/h $, both inequalities are satisfied for all $0 \leq j
\leq m$ if they are satisfied for $j=0$. Recallingl that $\epsi=
h^{-1/3}/16 \leq 1$, this in turn follows from
\begin{equation}\label{eq:assump1ab}
4h(2m+n) \leq  \epsi , \ \ \ \
h(2m+n) \leq \dfrac{1}{2}.
\end{equation}
If $h \leq \frac{1}{64}(m+n)^{-3}$, then
\begin{equation*}
\dfrac{8\delta}{ (m+n)^2} \leq \frac{\delta^{1/3}}{ (m+n)},
\end{equation*}
which implies \eqref{eq:assump1ab}.  We conclude that
\eqref{eq:assump1} holds, hence also \eqref{eq:assepsi}, and  hence we
can apply Lemma \ref{lem:estE} to 
$E_\epsi(h,m+N,m-j)$, $0 \leq j \leq m$.

\subsection*{Final estimate on $h$.} We first start by simplifying ${A_1}$. Lemma \ref{lem:estE} shows that
\[ 
E_\epsi\left(h,m+N,0\right) \lesssim (n+m+N)^{2n} \epsi^{-2} e^{-{N\epsi^2}/{16} }
\lesssim  t({3}/{h})^{2n+1} e^{- h^{-1/3}/16} .
\]
Thus 
\begin{equation}\label{eq:est(ii)}
\disp{{A_1} \lesssim \lambda(f) (1-2\epsi)^m ({3}/{h})^{2n+1} e^{- h^{-1/3}/16}.}
\end{equation}

To treat $ A_2 $ we note that 
\begin{equation*}
A_2 = \disp{\Lambda\left(f\right) \sum_{j=1}^m \dfrac{ \left(
    nm^2h\right)^j}{j!}  \left(1+2\epsi\right)^{m-j}} \leq \disp{
  \Lambda\left(f\right) \left(1+2\epsi\right)^m \left(e^{nm^2 h}   -1\right)} .
\end{equation*}
But since $h \leq (n+m)^{-3}$, $nm^2 h \leq 1$, and thus $\exp
\left(nm^2 h\right)  -1 \lesssim nm^2h$, and 
\begin{equation}\label{eq:est(iii)}
{A_2} \lesssim \Lambda\left(f\right) \left(1+2\epsi\right)^m nm^2h.
\end{equation}
We finally treat ${A_3}$. For that, we need the estimate on
$E_\epsi\left(h,m+N,m-j\right)$ proved in 
Lemma \ref{lem:estE}:
\begin{align*}
\disp{E_\epsi\left(h,m+N,m-j\right) }&\lesssim   \disp{  h^{m-j}
 (N+n+2m-j)^{2n+m-j} \epsi ^{-2}  
e^{-{(m+N)\epsi^2}/{16}  } } 
\\
      & \lesssim \disp{ {(3N)^{2n}}\epsi^{-2} (3hN)^{m-j}
          e^{-{(m+N)\epsi^2}/{16} }   }
\\
        & \lesssim \disp{  \left({3}/{h}\right)^{2n+1} 3^m
          e^{-h^{-1/3}/16}    
}  .
\end{align*}
Inserting this in the definition of ${A_3}$, 
\begin{equation*}
A_3 = \disp{\Lambda\left(f\right) \sum_{j=0}^m \dfrac{
    \left(nm^2h\right)^j }{j!} E_\epsi\left(h,m+N,m-j\right)}.
\end{equation*}
this gives
\begin{align*}
{A_3} & \lesssim  \disp{\Lambda\left(f\right)\sum_{j=0}^m
\dfrac{
    \left(nm^2h\right)^j }{j!} \left({3}/{h}\right)^{2n+1}
  3^m            e^{-h^{-1/3}/16} }   \\
 & \lesssim \disp{ \Lambda(f) \left({3}/{h}\right)^{2n+1} 3^m
   e^{-h^{-1/3}/16 } } .
\end{align*}
Here we used again $nm^2 h \leq 1$. Thus we get:
\begin{equation}\label{eq:est(iv)}
{A_3} \lesssim \disp{ 3^m \left({3}/{h}\right)^{2n+1}
  e^{-h^{-1/3}/16 }  }. 
\end{equation}

We recall that we are looking for $h_0$ such that for $ h < h_0 $, 
\begin{equation}\label{eq:18}
\lambda(f) (1-2\epsi)^m \geq {A_1} + {A_2} + {A_3}
\end{equation}
is satisfied. 
In view of \eqref{eq:est(ii)}, \eqref{eq:est(iii)},
\eqref{eq:est(iv)},  to obtain \eqref{eq:18} it is sufficient to have 
\begin{equation}\label{eq:19ii}
\lambda(f) (1-2\epsi)^m \geq 3 \lambda(f) (1-2\epsi)^m
\left({3}/{h}\right)^{2n+1}           e^{-h^{-1/3}/16} ,
\end{equation}
\begin{equation}\label{eq:19iii}
\lambda(f) (1-2\epsi)^m \geq 3 \Lambda(f) (1+2\epsi)^m nm^2 h ,
\end{equation}
\begin{equation}\label{eq:19iv}
\lambda(f) (1-2\epsi)^m \geq 3 \Lambda(f) 3^m \left({3}/{h}\right)^{2n+1}             e^{-h^{-1/3}/16}.
\end{equation}

Since $h \leq \delta=1/64$, $\epsi \leq 1/4$ and then $(1-2\epsi)^m \geq 10^{-m}$; moreover 
\begin{equation*}
\left( \dfrac{1+2\epsi}{1-2\epsi} \right)^m \leq \left( 1+ 8\epsi \right)^m \leq \left( 1+ \dfrac{8}{m} \right)^m \lesssim 1.
\end{equation*}
Thus \eqref{eq:19iii}, \eqref{eq:19iv} can be changed in
\begin{equation}\label{eq:22iii}
\lambda(f) \geq 3 \Lambda(f) nm^2 h ,
\end{equation}
\begin{equation}\label{eq:22iv}
\lambda(f) \geq 3 \Lambda(f) 30^m \left({3}/{h}\right)^{2n+1}
e^{-h^{-1/3}/16} .
\end{equation}

 Since $\lambda(f) \leq \Lambda(f)$, \eqref{eq:19ii} and \eqref{eq:22iv} are both implied by
\begin{equation}\label{eq:20}
\dfrac{\lambda(f)}{\Lambda(f)} \geq 3 \cdot 30^m
\left({30}/{h}\right)^{2n+1}             e^{-h^{-1/3}/16} .
\end{equation}
The logarithmic version of this inequality is
\begin{equation*}
\log\left( \dfrac{\lambda(f)}{\Lambda(f)} \right) \geq \log(3) +
(m+2n+1) \log(30) -(2n+1) \log(h) - h^{-1/3}/16, 
\end{equation*}
and thus taking $h \lesssim \log\left( \Lambda(f) / \lambda(f)
\right)^{-3} (m+n)^{-3}\log(n)^{-3}$ 
assures its validity. Indeed, 
\begin{equation}\label{eq:111}
h \lesssim \log \left( \dfrac{\Lambda(f)}{\lambda(f)} \right)^{-3} \
\Longrightarrow \ \ \log\left( \dfrac{\lambda(f)}{\Lambda(f)} \right) \gtrsim -h^{-1/3}
\end{equation}
and
\begin{equation}\label{eq:112}
h \lesssim \left( n \log (n) \right)^{-3} \Rightarrow n \log(h) \gtrsim -h^{-1/3}.
\end{equation}

The estimate \eqref{eq:22iii} is straightforward: we need
\begin{equation}\label{eq:113}
h \lesssim \dfrac{\lambda(f)}{\Lambda(f)} n^{-1}m^{-2}.
\end{equation} 

Let us chose 
\begin{equation*}
h \lesssim  \min\left(\dfrac{\lambda(f)}{ \Lambda(f)}, \log\left( \dfrac{\Lambda(f)}{\lambda(f)} \right)^{-3} \right) (m+n)^{-3}\log(n)^{-3}.
\end{equation*}
Then $h$ satisfies the three necessary conditions for Theorem
\ref{thm:boundN} to hold: \eqref{eq:111}, \eqref{eq:112}, and
\eqref{eq:113}. The bound on $N = 1/h$ is then given by
\begin{equation*}
N \gtrsim  \max \left( \log\left( \dfrac{\Lambda(f)}{\lambda(f)} \right)^{3}, \dfrac{\Lambda(f)}{ \lambda(f)} \right) (m+n)^3\log^3 n
\end{equation*}
which is the same as
\begin{equation*}
N \gtrsim  \dfrac{\Lambda(f)}{ \lambda(f)} (m+n)^3\log^3 n.
\end{equation*}

\vspace{0.5cm}
\begin{center}
\noindent
{\sc  Appendix: {A non-stationary phase lemma}} 
\end{center}
\vspace{0.4cm}
\renewcommand{\theequation}{A.\arabic{equation}}
\refstepcounter{section}
\renewcommand{\thesection}{A}
\setcounter{equation}{0}

We prove Lemma \ref{lem:estJ}. 
Let $\varphi\left(t\right)=-\log\left(t\right)+t$. Then $\varphi$ is a
one to one mapping on $(0,1]$ and on $[1,\infty)$. Let us then
consider the following integrals:
\begin{equation*}
J^-\left(\rho,\delta\right) = \disp{\int_0^{1-\delta}
  e^{\rho\left(\log\left(t\right)-t\right)} dt},  \ \ \ 
J^+\left(\rho,\delta\right) = \disp{\int_{1+\delta}^\infty e^{\rho\left(\log\left(t\right)-t\right)} dt}.
\end{equation*}
The change of variable $\varphi\left(t\right)=x$ gives
\begin{equation*}
J^-\left(\rho,\delta\right) = \disp{\int_{c^-}^\infty e^{-\rho x} \left(\dfrac{1}{\varphi^{-1}\left(x\right) }-1\right)^{-1} dx},
\end{equation*}
with $c^- = \varphi\left(1-\delta\right)$. Thus we need estimates on
$\varphi^{-1}\left(x\right)$. But on $(0,1-\delta]$, we have
$\varphi\left(t\right) \leq 1-\delta - \log\left(t\right)$. It implies
$\varphi^{-1}\left(x\right) \leq e^{1-\delta -x}$. This gives
\begin{equation*}
J^-\left(\rho,\delta\right) \leq \disp{\int_{c^-}^\infty \dfrac{e^{-\rho x}}{ e^{x-1+\delta} -1 }  dx}.
\end{equation*}
A lower bound for $e^{x-1+\delta} -1$ is 
\begin{equation*}
e^{x-1+\delta} -1 \geq \left( e^{-1+\delta}-e^{-c^-} \right) e^x \geq \delta e^{-1+\delta+x}. 
\end{equation*}
and hence
\begin{equation}\label{est1}
\begin{split} 
J^-\left(\rho,\delta\right) & \leq \disp{\int_{c^-}^\infty
  \dfrac{e^{1-\delta}}{\delta}e^{-\left(\rho +1\right)x}   dx =
  \dfrac{1-\delta}{\delta \left(\rho+1\right)}
  \left(\left(1-\delta\right) e^{-1+\delta}\right)^\rho }
\\ & \leq { \dfrac{1}{\rho \delta} \left(\left(1-\delta\right)
  e^{-1+\delta}\right)^\rho}.
\end{split} 
\end{equation}

 The same change of variable applied to  $J^+$ gives
\begin{equation*}
J^+\left(\rho,\delta\right) = \disp{\int_{c^+}^\infty e^{-\rho x} \left(1-\dfrac{1}{\varphi^{-1}\left(x\right) }\right)^{-1} dx}
\end{equation*}
with $c^+ = \varphi\left(1+\delta\right)$. On $(1+\delta, \infty)$, we have $\varphi\left(t\right) \leq t$ and then $\varphi^{-1}\left(x\right) \geq x$.
\begin{equation*}\label{eq:2}
\begin{split} 
J^+\left(\rho,\delta\right)  \leq \disp{\int_{c^+}^\infty e^{-\rho x}
  \left(1-\dfrac{1}{x}\right)^{-1} dx 
\leq = \dfrac{c^+}{c^+-1}\int_{c^+}^\infty  e^{-\rho x}dx}.
\end{split}
\end{equation*}
Since $\delta < 1$,
\begin{equation*}
\dfrac{c^+}{c^+-1} = \dfrac{\varphi(1+\delta)}{\varphi(1+\delta)-1} \lesssim \dfrac{1}{\delta^2}
\end{equation*}
and thus
\begin{equation}\label{est2}
J^+\left(\rho,\delta\right) \lesssim  \dfrac{1}{\rho \delta^2} \left( \left(1+\delta\right) e^{-1-\delta} \right)^\rho.
\end{equation}

Now, 
\begin{equation*}
\left(1-\delta\right)e^{-1+\delta}\leq \left(1+\delta\right)e^{-1-\delta}, \ \ \delta^2 \leq \delta,
\end{equation*}
and hence the estimates \eqref{est1} and \eqref{est2} give
\begin{equation*}
J\left(\rho,\delta\right) = J_- \left(\rho,\delta\right)
+ J_+ \left(\rho,\delta\right)
\lesssim \dfrac{1}{\rho \delta^2} \left(\left(1+\delta\right)e^{-1-\delta}\right)^\rho .
\end{equation*}
Also, 
\begin{equation*}
\left(1+\delta\right)e^{-\delta} \leq e^{ -{\delta^2}/{4} }, 
\end{equation*}
so that finally
\begin{equation*}
J\left(\rho,\delta\right) \lesssim \dfrac{1}{\rho \delta^2} \exp \left(-\rho\left(1+\dfrac{\delta^2}{4}\right)\right)  .
\end{equation*}

\end{document}